\documentclass[a4paper,11pt]{amsart}

\usepackage{graphicx, amssymb, bm}  
\usepackage{mathrsfs}   
\usepackage[numbers, sort&compress]{natbib}
\usepackage{url,hyperref} 
\usepackage[inner=2.5cm,outer=2.5cm,top=2.5cm,bottom=2.5cm,includehead,includefoot]{geometry}
\usepackage{float}
\hypersetup{citecolor=red, linkcolor=blue, colorlinks=true}

\newtheorem{thm}{Theorem}[section]

\newtheorem{lemma}[thm]{Lemma}

\theoremstyle{definition}

\newtheorem{remark}[thm]{Remark}
\newtheorem{hypothesis}[thm]{Hypothesis}
\newtheorem{example}[thm]{Example}

\renewcommand\leq{\leqslant} 
\renewcommand\geq{\geqslant}

\DeclareMathOperator{\Out}{Out}
\DeclareMathOperator{\Inn}{Inn}
\DeclareMathOperator{\Aut}{Aut}
\DeclareMathOperator{\fix}{fix}
\DeclareMathOperator{\core}{core}
\DeclareMathOperator{\soc}{soc}
\DeclareMathOperator{\twr}{twr}
\DeclareMathOperator{\Sym}{Sym}

\DeclareMathOperator{\GL}{GL}
\DeclareMathOperator{\SL}{SL}
\DeclareMathOperator{\PSL}{PSL}
\DeclareMathOperator{\Sp}{Sp}

\title{Bases of twisted wreath products}
\author{Joanna B. Fawcett}
\dedicatory{Dedicated to the memory of my doctoral supervisor Jan Saxl}

\address[Fawcett]{Department of Mathematics,  Imperial College London, South Kensington Campus, 
London, SW7 2AZ, United Kingdom}
\email{j.fawcett@imperial.ac.uk}


\thanks{The author thanks the anonymous reviewers for their  valuable comments.   
}


\begin{document}

\maketitle

\begin{abstract}
We study the base sizes of finite quasiprimitive permutation groups of \textit{twisted wreath type}, which are precisely the finite  permutation groups  with a unique  minimal normal subgroup that is also non-abelian,    non-simple and regular. Every permutation  group of twisted wreath type is permutation isomorphic to a twisted wreath product $G=T^k{:}P$ acting on its base group $\Omega=T^k$, where $T$ is some non-abelian simple group and $P$ is some group acting  transitively on  $\bm{k}=\{1,\ldots,k\}$ with $k\geq 2$. 
We prove that if  $G$ is primitive on $\Omega$ and $P$ is  quasiprimitive on~$\bm{k}$, then $G$ has base size $2$. We also  prove that the proportion of pairs of points  that are bases for $G$ tends to 1 as $|G|\to \infty$ when $G$ is primitive on $\Omega$ and $P$ is primitive on~$\bm{k}$. Lastly, we determine  the base size of any quasiprimitive group of twisted wreath type  up to four possible values (and three in the primitive case). In particular, we demonstrate that there are many families of primitive groups of twisted wreath type with arbitrarily large base sizes. 
\end{abstract}

\section{Introduction}
\label{s:intro}

Bases are a fundamental tool in permutation group theory and are used extensively in computational group theory (see~\cite{Ser2003}). For a  permutation group $G$ on a set $\Omega$, a \textit{base}  is a  subset of $\Omega$ whose pointwise stabiliser in $G$ is trivial. The \textit{base size} of $G$, denoted by $b_\Omega(G)$ or $b(G)$, is the  minimal cardinality of a base for $G$. It is immediate from the definition of a base that  $|G|\leq |\Omega|^{b(G)}$, so $b(G)$ was first studied in order to  bound the order of a finite primitive permutation group (e.g.,~\cite{Boc1889,Bab1981}). Increasingly, however, attention has focused on the parameter $b(G)$ itself, and with the advent of the classification of the finite simple groups (CFSG), there have been significant advances in our understanding of the base sizes of finite primitive groups,   several of which are outlined below. 

This paper concerns bases of twisted wreath products. By the O'Nan-Scott theorem (see~\cite{LiePraSax1988} or~\S\ref{s:prelim}), any finite primitive permutation group with a unique non-abelian regular minimal normal subgroup is  permutation isomorphic to a twisted wreath product  with a prescribed action (see~\S\ref{s:tw} for precise definitions).  This  class of primitive groups was not well understood until Baddeley's definitive work~\cite{Bad1993} on the subject, and even now it remains somewhat mysterious. 

For our purposes, a  twisted wreath product $G$ is some semidirect product $B{:}P$ with  $B\simeq T^k$ such that $T$ is a finite non-abelian simple group and $P$ is a finite group acting transitively  on $\bm{k}:=\{1,\ldots,k\}$ where $k\geq 2$. The base group $B$ is the unique minimal normal subgroup of $G$, so  $B=\soc(G)$, the socle of $G$. Further, $G$ acts faithfully and transitively on $B$, with $B$ acting regularly by right multiplication, and $P$ acting by conjugation. In particular, $G$ is  quasiprimitive. (A permutation group is 
 \textit{quasiprimitive} if every non-trivial normal subgroup is transitive. For example, every primitive permutation group is quasiprimitive.)
There are two main issues when working with twisted wreath products. The first is that given some $T$ and $P$ as above, it can be difficult to determine whether a twisted wreath product $G$ exists. The second is that given some twisted wreath product $G$, the conditions on $T$ and $P$ for  $G$ to be primitive are quite subtle; necessarily, the action of $P$ on~$\bm{k}$ is faithful, and $T$ is a composition factor of $P_1$, the stabiliser  in $P$ of $1\in \bm{k}$. Both issues are addressed in~\cite{Bad1993}. 
 
One of the key advances in our understanding of base sizes concerns   almost simple groups. Let $\mathcal{C}$ denote the class of finite almost simple primitive permutation groups, but exclude the \textit{standard} actions, which, roughly speaking, are those  arising  from the alternating group $A_n$ on a set of subsets or partitions of $\bm{n}$, or a classical group on a set of subspaces of the natural module (see~\cite[Definition~1.1]{Bur2007}). The base size of a standard permutation group can be arbitrarily large. 
However, Cameron  conjectured~\cite{Cam1992} that there is an absolute constant $c$ such that $b(G)\leq c$ for all $G\in\mathcal{C}$. This conjecture was proved by Liebeck and Shalev~\cite{LieSha1999}, and, remarkably, it was later established that $c=7$ is best possible~\cite{Bur2007,BurLieSha2009,BurObrWil2010,BurGurSax2011}. 
 Jan Saxl then initiated an  ambitious project to classify the primitive permutation groups with base size $2$.   Much progress has been made on this problem for primitive groups of almost simple type~\cite{BurGurSax2011,JamJ2006,BurObrWil2010,BurGurSax2014,JamJ20062,MulNeuWil2007,Bur2021}, diagonal type~\cite{Faw2013} and affine type~\cite{Goo2000,KohPah2001,FawObrSax2016,FawMulObrWil2019,Lee2021,Lee2021-2}.  Our first main result establishes  that a large class of primitive twisted wreath products have base size $2$.

First we require some terminology.   
A finite quasiprimitive permutation group $G$ is of \textit{twisted wreath type} if $G$ has a unique non-abelian non-simple regular minimal normal subgroup (in which case $G$ is permutation isomorphic to a twisted wreath product by~\cite{Pra1993}). The \textit{top group}  of $G$ is the permutation group induced by the conjugation action of $G$ on the simple factors of $\soc(G)$. If $G$ is a twisted wreath product $B{:}P$ as described above, then $G$ has top group $P^{\bm{k}}$, and if $G$ is primitive on $B$, then $P$ is faithful on $\bm{k}$ (see~\S\ref{s:tw}).

\begin{thm}
\label{thm:QP} 
Let $G$ be a primitive permutation group of twisted wreath type. If the top group of $G$ is quasiprimitive, then $b(G)=2$.
\end{thm}

In fact, Theorem~\ref{thm:QP} is a consequence of a more technical result,  Theorem~\ref{thm:QPgeneral}, 
where we prove that many quasiprimitive permutation groups of twisted wreath type have base size $2$.

\begin{remark}
\label{remark:mintwisted}
To provide some context for the assumption in Theorem~\ref{thm:QP} that the top group is quasiprimitive, we appeal to a result of Baddeley~\cite{Bad1993} (see Theorem~\ref{thm:blowup}). He proves that if $G$ is a primitive permutation group of twisted wreath type, then either $G$ is of minimal-twisted type---in which case the top group  is  quasiprimitive---or  
$G$ is a blow-up  of a primitive permutation group $H$ of minimal-twisted type, in which case $\soc(H)^r\unlhd  G\leq H\wr S_r$  where  $G$ acts on $\soc(H)^r$ via the product action and $r\geq 2$.
This gives us the following version of the O'Nan-Scott theorem: any primitive permutation group with a non-abelian socle is  either of almost simple, diagonal or minimal-twisted type, or  a blow-up of one of these types. While very little is known about the  primitive blow-ups with base size $2$, much is known about the primitive groups of almost simple or diagonal type with base size $2$. Thus Theorem~\ref{thm:QP} contributes to this picture by showing that every primitive  group of minimal-twisted  type has base size $2$. Moreover, we use Baddeley's result and Theorem~\ref{thm:QP} to prove the primitive case of Theorem~\ref{thm:pyber} below. 
However, we caution the reader that the converse of Theorem~\ref{thm:QP} does not hold (see Examples~\ref{exp:topnotQP} and~\ref{exp:SkwrSr}).
\end{remark}

While we are unable to completely classify the primitive groups of twisted wreath type with base size $2$, we instead
 determine the base size of any  quasiprimitive group  of twisted wreath type up to four possible values (and three in the primitive case). Our result is closely related to another significant   advance in the theory of base sizes. Recall that if $G$ is  a permutation group  of degree $n$, then $|G|\leq n^{b(G)}$, so $\log_n|G|\leq b(G)$. Pyber conjectured~\cite{Pyb1993} that there is an absolute constant $c$ such that $b(G)\leq c\log_n|G|$ for every finite  primitive group $G$ of degree~$n$. Building on the work of 
\cite{Ben05,LieSha1999,LieSha2002,LieSha2014,Faw2013,GluSerSha1998,BurSer2015,FawPra2016}, this  conjecture was  recently established by Duyan, Halasi and Mar\'{o}ti~\cite{DuyHalMar2018}. Halasi, Liebeck and Mar\'{o}ti~\cite{HalLieMar2019} then proved  that  $b(G)\leq 2(\log_n|G|)+24$, which is asymptotically best possible. One of the key ingredients in the proof of Pyber's conjecture concerns the \textit{distinguishing number} of a permutation group $G$ on $\Omega$,  which is the minimal number of parts in a partition  of $\Omega$  for which only the identity fixes every part; we denote it by $d_\Omega(G)$ or $d(G)$.  In~\cite{DuyHalMar2018}, it is proved that if $G$ is  a transitive permutation group of degree $n>1$, then  $\sqrt[n]{|G|}<d(G)\leq 48 \sqrt[n]{|G|}$. We use this theorem to establish our next two results. 

\begin{thm}
\label{thm:pyber}
Let $G$ be a quasiprimitive permutation  group on $\Omega$ of twisted wreath type with socle $T^k$ and top group $R$,  where $T$ is a simple group. Then
$$
b_\Omega(G)=\left\lceil \frac{\log{|G|}}{\log {|\Omega|}}\right\rceil +\varepsilon=\left\lceil \frac{\log{d_{\bm{k}}(R)}}{\log {|T|}}\right\rceil+\delta
$$
for some $\varepsilon,\delta\in \{0,1,2,3\}$ where $\varepsilon\leq \delta$, and if $G$ is primitive on $\Omega$, then $\varepsilon\neq 3$ 
and $\delta\leq \varepsilon+1$.
\end{thm}

For any non-abelian simple group $T$ and transitive permutation group $R$ of degree $k\geq 2$, there is a quasiprimitive group of twisted wreath type with socle $T^k$ and top group $R$ (see Remark~\ref{remark:lotsQP}). Since the symmetric group $S_r$ has distinguishing number $r$,  Theorem~\ref{thm:pyber} demonstrates that the base size of a quasiprimitive group of twisted wreath type can be arbitrarily large. In fact, we will see that this is also true for primitive groups of twisted wreath type. 

In our next result, we establish nearly exact base size formulae for a large class of primitive permutation groups of twisted wreath type.

\begin{thm}
\label{thm:exp}
Let $H$ be a primitive permutation group on $\Delta$ of twisted wreath type with socle $T^m$, where $T$ is a  simple group. Let $K$ be a transitive subgroup of $S_r$ where $r\geq 2$. Then $G:=H\wr K$ is a primitive permutation group on $\Omega:=\Delta^r$ of twisted wreath type. If the top group of $H$ is quasiprimitive, then 
$$
b_\Omega(G)=\left\lceil \frac{\log|G|}{\log|\Omega|}\right\rceil+\varepsilon=\left\lceil \frac{\log{d_{\bm{r}}(K)}}{m\log {|T|}}\right\rceil+\delta
$$
 for some $\varepsilon\in \{0,1\}$ and  $\delta\in \{1,2\}$ where $\varepsilon+1\leq \delta$. 
\end{thm}

Thus the base size of a primitive group of twisted wreath type can be arbitrarily large.
Note that for any non-abelian simple group $T$, there are primitive groups of twisted wreath type with quasiprimitive top groups  and socle $T^m$ for infinitely many $m$ (see Example~\ref{exp:diag}), so  the base size of a primitive group of twisted wreath type can grow in various ways (see also Example~\ref{exp:SkwrSr}).

Our final result is related to a probabilistic version of the Cameron conjecture. 
 Recall the class $\mathcal{C}$ of almost simple groups that was defined above. Cameron and Kantor  conjectured~\cite{CamKan1993} that there is an absolute constant $c'$ such that the proportion of $c'$-tuples of points that are (ordered) bases for $G$ tends to 1 as $|G|\to \infty$ for $G\in\mathcal{C}$. This conjecture was  proved by Liebeck and Shalev~\cite{LieSha1999}, and it was later shown that  $c'=6$ is best possible~\cite{Bur2018,BurLieSha2009,LieSha2005}. Analogous results have been  established for primitive groups of diagonal type~\cite{Faw2013} and coprime primitive linear groups~\cite{DuyHalPod2020}. 
 The following provides a partial result for the twisted wreath case. 

\begin{thm}
\label{thm:prob}
For those primitive permutation groups  $G$ of twisted wreath type whose  top groups are  primitive,  the proportion of pairs of points  that are bases for $G$ tends to 1 as $|G|\to \infty$.
\end{thm}

This paper is organised as follows.
In~\S\ref{s:prelim}, we provide some basic notation and terminology, along with a brief description of the O'Nan-Scott theorem. In~\S\ref{s:dist}, we describe some of the bounds on   distinguishing numbers from the literature. In~\S\ref{s:tw}, we outline the theory of twisted wreath products and state a version of Baddeley's classification of the primitive permutation groups of twisted wreath type. Theorems~\ref{thm:QP},~\ref{thm:exp} and~\ref{thm:prob} are all consequences of more technical results that concern much larger classes of groups. In~\S\ref{s:prob}, we prove Theorem~\ref{thm:prob} and its more general form. In~\S\ref{s:bounds}, we give some general bounds on the base size of a twisted wreath product. In~\S\ref{s:QP}, we  prove Theorem~\ref{thm:QP} and its more general form, and in~\S\ref{s:pyber}, we prove Theorem~\ref{thm:pyber} as well as Theorem~\ref{thm:exp} and its more general form.  Lastly, we give some examples in~\S\ref{s:exp}.

Note that some of the material in this paper can be found in the author's Ph.D. thesis~\cite{FawThesis}. However,  
the results in~\S\ref{s:bounds}--\ref{s:QP} are more general than those  in~\cite{FawThesis}, and~\S\ref{s:pyber}--\ref{s:exp} are new.

\section{Preliminaries}
\label{s:prelim}

All groups in this paper are finite, and all actions are written on the right. Basic terminology for permutation groups may be found in~\cite{DixMor1996}.   For a subgroup $H$ of a group $G$, the \textit{core} of $H$ in $G$, denoted by $\core_G(H)$, is defined to be $\bigcap_{g\in G}H^g$; it is precisely the kernel of the action of $G$ on the (right)  cosets of $H$.  For a group $G$ acting on a set $\Omega$, we write $G^\Omega$ for the permutation group induced by $G$ on $\Omega$.  For $g\in G$, we write 
 $\fix_\Omega(g)$ for the set of fixed points of $g$. We also write $\omega_\Omega(G)$ for the number of orbits of $G$ on $\Omega$, and we write  $\omega_\Omega(g)$ for $\omega_\Omega(\langle g\rangle)$. 

Next we give some notation for wreath products and their actions. Let $H$ be a permutation group on $\Delta$, and let $K\leq S_r$. The wreath product $H\wr K$ is the semidirect product $H^r{:}K$, where  $\pi^{-1}\in K$  maps $(h_1,\ldots,h_r)\in H^r$ to $(h_{1^\pi},\ldots,h_{r^\pi})$. In the product action of $H\wr K$ on $\Delta^r$, the elements $(h_1,\ldots,h_r)\in H^r$ and $\pi^{-1}\in K$ map $(\delta_1,\ldots,\delta_r)\in \Delta^r$ to $(\delta_1^{h_1},\ldots,\delta_r^{h_r})$ and $(\delta_{1^\pi},\ldots,\delta_{r^\pi})$, respectively. In the imprimitive action of $H\wr K$ on $\Delta\times \bm{r}$, the element $(h_1,\ldots,h_r)\pi\in H\wr K$ maps $(\delta,i)\in \Delta\times \bm{r}$ to $(\delta^{h_i},i^\pi)$. 

We now define the blow-up of a permutation group, following~\cite[\S2]{Bad1993}. Let $W:=H\wr S_r$, with $H$ and $r$ as above. Let $\rho:W\to S_r$ be the projection map, and let $W_0$ be the preimage of the stabiliser in $S_r$ of the point $r\in \bm{r}$, so that $W_0=(H\wr S_{r-1})\times H$. Let $\sigma:W_0\to H$ be the projection map.  A subgroup $G$ of $W$ is \textit{large} if $\rho(G)$ is transitive on $\bm{r}$ and $\sigma(G\cap W_0)=H$.
A permutation group $X$ is a \textit{blow-up of $H$ on $\Delta$ of index $r$} if $X$ is permutation isomorphic to a large subgroup $G$ of $H\wr S_r$ where $\soc(H)^r\leq G$ 
and $G$ acts via the product action on $\Delta^r$.

To finish this section, we briefly consider the structure of a quasiprimitive permutation group. 
By the O'Nan-Scott theorem~\cite{LiePraSax1988}, the finite primitive permutation groups are of affine type, almost simple type, diagonal type, product type or twisted wreath type. A similar theorem is proved in~\cite{Pra1993} for the finite quasiprimitive permutation groups. In Table~\ref{tab:ONS}, we give a rough description of these types (using a similar format to the table   in~\cite{BurSer2015}).   As in the primitive case, a quasiprimitive group  $G$ has at most two minimal normal subgroups, and when there are two, they are isomorphic and regular. The socle $B$ of  $G$ is therefore isomorphic to $T^k$ for some simple group $T$ and $k\geq 1$, and $T$ is abelian if and only if $G$ is of affine type. Note that in type III(b)(ii), the action of $G$ is the usual product action, but this is not necessarily the case in type III(b)(i) when $G$ is not primitive. Note also that in type II, the socle $T$ may be regular when $G$ is quasiprimitive,  but $T$ is not regular when $G$ is primitive. However, the socle $T^k$ is not regular in types III(a) or III(b). 

 \begin{table}[!h]
\centering
\begin{tabular}{ l l }
\hline
Type & Description  \\
\hline
I & Affine type: $T$ abelian \\
II & Almost simple type: $T\leq G\leq \Aut(T)$ \\
III(a)(i) & Diagonal type: $G\leq T^k.(\Out(T)\times P)$, $P\leq S_k$ transitive, $k\geq 2$ \\
III(a)(ii) & Diagonal type: $G\leq T^2.\Out(T)$, $k=2$\\
III(b)(i) & Product type: $G\leq H\wr P$, $H$ quasiprimitive of type II, $P\leq S_k$ transitive, $k\geq 2$ \\
III(b)(ii) & Product type: $G\leq H\wr P$, $H$ quasiprimitive of type III(a), $P\leq S_\ell$ transitive, $\ell\mid k$ \\
III(c) & Twisted wreath type: $G=T^k{:}P$, $P$ acts transitively on $\bm{k}$, $k\geq 2$ \\
\hline
\end{tabular}
\caption{The finite quasiprimitive groups $G$ with socle $T^k$}
\label{tab:ONS}
\end{table}

\section{Bounds on distinguishing numbers}
\label{s:dist}

In the introduction, we defined the distinguishing number of a permutation group. Since this concept is critical to our study of the base sizes of quasiprimitive groups of twisted wreath type, we give a brief survey of some of the known bounds on distinguishing numbers. For more details about the origin of this parameter and its uses in graph theory, see the survey paper~\cite{BaiCam2011}.

First we need some terminology. 
Let $G$ act on a set $\Omega$. We say that a partition of $\Omega$ is \textit{distinguishing} if the only elements of $G$ that fix every part of the partition are those that fix every element of $\Omega$. The \textit{distinguishing number} of $G$, denoted by $d_\Omega(G)$ or $d(G)$, is the minimal number of parts in a distinguishing partition for $G$. 
Similarly, we say that a subset of $\Omega$ is a \textit{base} for $G$ if the only elements of $G$ that fix the subset pointwise are those that fix $\Omega$ pointwise, and we denote the minimal size of a base for $G$ by $b_\Omega(G)$ or $b(G)$. Observe that $d(G)\leq b(G)+1$. 
Observe also that $d_{\bm{n}}(S_n)=b_{\bm{n}}(S_n)+1=n$ and $d_{\bm{n}}(A_n)=b_{\bm{n}}(A_n)+1=n-1$. 

Let $G$ be a permutation group on $\Omega$ with degree $n$. Let $\mathcal{P}_d(\Omega)$ be the set of ordered partitions of $\Omega$ with $d$ (possibly empty) parts where $d\geq 1$. 
 Since $G$  has a regular orbit on $\mathcal{P}_{d(G)}(\Omega)$, it follows that $|G|<d(G)^n$ when $G\neq 1$. The following remarkable  result is \cite[Theorem~1.2]{DuyHalMar2018}.

\begin{thm}
\label{thm:dist}
Let $G$ be a transitive permutation group  of degree $n>1$. Then  $$\sqrt[n]{|G|}<d(G)\leq 48 \sqrt[n]{|G|}.$$
\end{thm}

For certain classes of permutation groups, we can be much more precise.
 Let $G$ be a permutation group on $\Omega$ with degree $n$, and suppose that $G$ is not $A_n$ or $S_n$.
Cameron, Neumann and Saxl~\cite{CamNeuSax1984} proved that if $G$ is primitive, then either  
 $G$ has a regular orbit on the power set of $\Omega$ (i.e., $d(G)=2$), or $G$ is one of finitely many exceptions. By work of  Seress~\cite{Ser1997},  the exceptions have degree at most $32$, and Dolfi later proved in~\cite[Lemma~1]{Dol2000} that $d(G)\leq 4$. 
Moreover,  these results were used by Duyan, Halasi and Mar\'oti in their proof of Theorem~\ref{thm:dist} to show that if $G$ is quasiprimitive, then  $d(G)\leq 4$; see~\cite[Lemma~2.7]{DuyHalMar2018}. In fact, Devillers, Morgan and Harper proved~\cite{DevMorHar2019} that if $G$ is quasiprimitive but not primitive, then $d(G)=2$. They also  proved that if $G$ is semiprimitive but not quasiprimitive, then $d(G)\leq 3$, with equality if and only if $G=\GL_2(3)$ in its action on non-zero vectors. (A permutation group $G$ is  \textit{semiprimitive} if every normal subgroup is transitive or semiregular.)  Furthermore, Gluck proved in \cite[Corollary~1]{Glu1983} that if $G$ has odd order, then $d(G)\leq 2$, and  Seress proved in~\cite[Theorem~1.2]{Ser1996} that if $G$ is soluble, then $d(G)\leq 5$ (and this bound is tight). More generally,  Halasi and Podoski used the work of Dolfi~\cite{Dol2000}  to prove in~\cite[Theorem~2.3]{HalPod2016} that, for $s\geq 3$, if $A_s$ is not a section of $G$, then $d(G)\leq s$.  We summarise these results in the following theorem for convenience.

\begin{thm}
\label{thm:summary}
Let $G$ be a non-trivial  permutation group  of degree $n$ that is not $A_n$ or $S_n$. 
\begin{itemize}
\item[(i)]If $G$ is semiprimitive, then either $n>32$ and $d(G)=2$, or $n\leq 32$ and $d(G)\leq 4$.
\item[(ii)]  If $G$ is soluble, then $d(G)\leq 5$, and if $G$ has odd order, then $d(G)= 2$.
\item[(iii)] For $s\geq 3$, if  $A_s$ is not a section of $G$, then $d(G)\leq s$.
\end{itemize}
\end{thm}

To finish this section, we discuss  the distinguishing numbers of imprimitive wreath products.  Chan proved in~\cite[Theorem~2.3]{Cha2006} that for any $H\leq \Sym(\Delta)$ and $K\leq S_r$,  the distinguishing number of $H\wr K$ in its imprimitive action on $\Delta\times\bm{r}$ is  the minimal $d$ for which $H$ has at least $d_{\bm{r}}(K)$ regular  orbits on  $\mathcal{P}_d(\Delta)$. Since  an orbit  is regular whenever it consists of distinguishing partitions, $d\geq d_\Delta(H)$. By  Dolfi~\cite[Remark~2]{Dol2000}, if $s\geq d_\Delta(H)$, then 
$H$ has at least $s+1$ regular orbits on $\mathcal{P}_{s+1}(\Delta)$, so  
 $d\leq \max\{d_\Delta(H)+1,d_{\bm{r}}(K)\}$. Moreover, the following result, which was used in~\cite{DuyHalMar2018}  to prove Theorem~\ref{thm:dist}, shows that $d$ can be  much smaller than $d_{\bm{r}}(K)$; see \cite[Lemma~2.4 and Remark~2.5]{DuyHalMar2018} for a proof. 

\begin{lemma}
\label{lemma:distblock}
Let $G$ be a transitive permutation group on $\Omega$. Let $\Delta$ be a block of $G$ with size $m$, and let $\Sigma:=\{\Delta^g: g\in G\}$. Let $H:=G_\Delta^\Delta$ and $K:=G^\Sigma$. Then $d_\Omega(G)\leq d_\Delta(H)\lceil \sqrt[m]{d_\Sigma(K)}\rceil$. 
\end{lemma}

\section{Twisted wreath products}
\label{s:tw}

We begin by defining a twisted wreath product, as in~\cite{Pra1993}. Let  $T$ be a finite non-abelian simple group, and let $P$ be a finite group acting transitively on $\bm{k}$ where $k\geq 2$. Let $Q:=P_1$, the stabiliser in $P$ of $1\in \bm{k}$, and let $\varphi:Q\to \Aut(T)$ be a homomorphism such that $\core_P(\varphi^{-1}(\Inn(T)))=1$.
 For any subset $X$ of $P$, let $T^X$ denote the set of functions from $X$ to $T$, and note that $T^X$ forms a group under pointwise multiplication with identity $e$ where $e(x):=1$ for all $x\in X$. Let 
 $$B:=\{f\in T^P \ |\ f(xq)=f(x)^{\varphi(q)}\ \mbox{for all}\ x\in P,q\in Q\}.$$
  Then $B\leq T^P$, and $P$ acts on $B$  by $f^z(x):=f(z x)$ for  $z,x\in P$ and $f\in B$. Since $(fg)^z=f^z g^z$ for all   $z\in P$ and $f,g\in B$, we may define the \textit{twisted wreath product} $T\twr_\varphi P$  to be the semidirect product $B{:} P$. Further, $T\twr_\varphi P$ acts on  $B$ by $f^{gx}:=(fg)^x$ for $f,g\in B$ and $x\in P$. Note that $P=(T\twr_\varphi P)_e$. We refer to $B$ as the \textit{base group} of $T\twr_\varphi P$.
  
Next we make some observations concerning the twisted wreath product  $G:=T\twr_\varphi P$ and its action on $B$. 
Since $P$ is transitive on $\bm{k}$, we may choose a left transversal $L:=\{a_1,\ldots,a_k\}$ for $Q$ in $P$ such that $i^{a_i}=1$ for all $i\in\bm{k}$.
Any element of $T^L$ can be extended to an element of $B$, so $B= T_1\times\cdots\times T_k\simeq T^k$ where $T_i:=\{f\in B : f(a_j)=1\ \mbox{for}\ j\in\bm{k}\setminus\{i\}\}$. Observe that if $x\in P$ and $i\in \bm{k}$, then $q_{x,i}:=a_i^{-1}xa_{i^x}\in Q$, so  for all $f\in B$,
\begin{equation}
\label{eqn:tw}
f^x(a_{i^x})=f(a_i)^{\varphi(q_{x,i})}.
\end{equation}
 It follows that for all $x\in P$ and $i\in \bm{k}$, 
\begin{equation}
\label{eqn:top}
T_i^x=T_{i^x}. 
\end{equation}
Hence $B$ is a minimal normal subgroup of $G$.  Since $1=\core_P(\varphi^{-1}(\Inn(T)))\simeq C_G(B)$  by~\cite[Proposition~2.7]{Bad1993}, 
$B=\soc(G)$. Further,  $1=\core_P(\ker(\varphi))=C_P(B)$ by~(\ref{eqn:tw}), so the action of $G$ on $B$ is faithful. Thus $G$ is a quasiprimitive permutation group whose socle $B$ is non-abelian, non-simple and  regular. Further, $P^{\bm{k}}$ is the top group of $G$ by~(\ref{eqn:top}).
  
We will use the following hypothesis to simplify notation for the remainder of the paper. 

\begin{hypothesis}
\label{hyp:tw}
Let  $T$ be a finite non-abelian simple group, and let $P$ be a finite group acting transitively on $\bm{k}$ where $k\geq 2$. Let $Q:=P_1$, the stabiliser in $P$ of $1\in \bm{k}$, and let $\varphi:Q\to \Aut(T)$ be a homomorphism such that $\core_P(V)=1$ where $V:=\varphi^{-1}(\Inn(T))$. Let $U:=\ker(\varphi)$.  Let $G:= T\twr_\varphi P$, and let $B$ be the base group of $G$. Recall that $G$ acts on  $B$ by $f^{gx}:=(fg)^x$ for $f,g\in B$ and $x\in P$.
Let $L:=\{a_1,\ldots,a_k\}$ be a left transversal for $Q$ in $P$ such that $i^{a_i}=1$ for all $i\in\bm{k}$.
For $x\in P$ and $i\in \bm{k}$, let  $q_{x,i}:=a_i^{-1}xa_{i^x}\in Q$. Recall that~(\ref{eqn:tw}) and~(\ref{eqn:top}) hold.
\end{hypothesis}

Recall from the introduction that a finite quasiprimitive permutation group $G$ is of twisted wreath type whenever $G$ has a unique non-abelian non-simple regular minimal normal subgroup. By the quasiprimitive version of the O'Nan-Scott theorem~\cite{Pra1993} (see~\S\ref{s:prelim}), such groups are precisely those described by Hypothesis~\ref{hyp:tw} (see~\cite[Remark~2.1]{Pra1993}). We record this result here. 

\begin{thm}
\label{thm:twONan}
 For a finite  permutation group $X$,  the following are equivalent.
\begin{itemize}
\item[(i)] $X$ is quasiprimitive   of twisted wreath type. 
\item[(ii)] $X$ is permutation isomorphic to a twisted wreath product $G$ acting on $B$ where $G$ and $B$ are  described by Hypothesis~\emph{\ref{hyp:tw}}.
\end{itemize}
Further, if (ii) holds, then $G$ has  socle $B\simeq T^k$ and top group  $P^{\bm{k}}$.
\end{thm}

\begin{remark}
\label{remark:lotsQP}
For any non-abelian simple group $T$ and transitive permutation group $P$ on~$\bm{k}$ where $k\geq 2$, we can construct a quasiprimitive permutation group of twisted wreath type
with socle $T^k$ and top group $P$. Let $Q$ be the stabiliser in $P$ of $1\in \bm{k}$, and  choose any homomorphism $\varphi:Q\to\Aut(T)$ (such as the trivial one). Since $\core_P(Q)=1$,  the conditions of Hypothesis~\ref{hyp:tw} are met,  so by  Theorem~\ref{thm:twONan},  $G:=T\twr_\varphi P$ has the desired properties.  
For example, if $\varphi$ is the trivial homomorphism, then $G$ is permutation isomorphic to the wreath product $T\wr P$ acting via the product action on $T^k$, where $T$ acts on $T$ by right multiplication. However, $G$ is not primitive  since $\{(t,\ldots,t):t\in T\}$ is a non-trivial block of $T\wr P$ on $T^k$ (cf. Lemma~\ref{lemma:twmax}).

There are also examples  where $P$ is not faithful on $\bm{k}$.  Let $R$ be any transitive permutation group on $\bm{k}$ where $k\geq 2$, and let $P:=\langle z\rangle\times R$ where $z$ is an involution that fixes $\bm{k}$ pointwise. Then $P$ is transitive on $\bm{k}$ and $\core_P(Q)=\langle z\rangle$ where  $Q:=P_1=\langle z\rangle \times R_1$. Let $\tau$ be a transposition in $S_n$ where $n\geq 5$, and define $\varphi:Q\to S_n$ by $z\mapsto \tau$ and $R_1\mapsto 1$, so that $\core_P(\varphi^{-1}(A_n))=1$. Then $G:=A_n\twr_\varphi P$ is a quasiprimitive permutation  group of twisted wreath type with socle $A_n^k$  and top group $R$. Observe that $S_n\wr R$ acts on $A_n^k$ via the product action, where $A_n$ acts on $A_n$ by right multiplication and $\langle\tau\rangle $ acts on $A_n$ by conjugation, and $G$ is permutation isomorphic to the subgroup $A_n\wr R\langle (\tau,\ldots,\tau)\rangle$. Similarly,  $S_n\wr R$ is a quasiprimitive permutation group of twisted wreath type with socle $A_n^k$, top group $R$ and point stabiliser $\langle\tau\rangle\wr R$.  
\end{remark}

Next we have some  basic but useful observations.

\begin{lemma}
\label{lemma:PonB}
Under Hypothesis~\emph{\ref{hyp:tw}},  $b_B(G)=b_B(P)+1$. 
\end{lemma}

\begin{proof}
Since $G$ is transitive on $B$, we may assume that a base of minimal size for $G$ contains the identity element $e$ of $B$. Since $P=G_e$, the result follows. 
\end{proof}
 
\begin{lemma}
\label{lemma:ideal}
Under Hypothesis~\emph{\ref{hyp:tw}}, $f(a_i)^{\varphi(q_{x,i})}=f(a_{i^x})$ for all $x\in P$, $i\in\bm{k}$ and $f\in \fix_B(x)$.  
\end{lemma}

\begin{proof}
This is immediate from~(\ref{eqn:tw}).  
\end{proof}

Our next result is~\cite[Theorem 3.5]{Bad1993}, which characterises when a twisted wreath product is primitive. 

\begin{thm}
\label{thm:primitivetwr}
\label{thm:twmax}
Assume Hypothesis~\emph{\ref{hyp:tw}}. Let  $M\unlhd P$ where $MU=MV$.  Let $U':=U\cap M$ and $V':=V\cap M$. Then  $G$ is primitive on $B$  if and only if   the following three conditions hold.
\begin{itemize}
\item[(i)] $V'/U'\simeq T$.

\item[(ii)] $Q=N_P(U')\cap N_P(V')$.

\item[(iii)] If $R$ is a subgroup of $M$ normalised by $Q$ for which $R\cap V'=U'$, then $R=U'$. 
\end{itemize}
\end{thm}

Note that we can always take the group $M$ in Theorem \ref{thm:twmax} to be $P$ itself, so this result does indeed  characterise when a twisted wreath product is primitive on its base group. 

One necessary condition of primitivity that is not immediate from Theorem~\ref{thm:twmax} is that $P$ acts faithfully on $\bm{k}$ (i.e., $\core_P(Q)=1$). This was first proved in~\cite[Case~2(a)]{LiePraSax1988}. We record this result below, along with a useful consequence of Theorem~\ref{thm:twmax}. 

\begin{lemma}
\label{lemma:twmax}
If Hypothesis~\emph{\ref{hyp:tw}} holds and  $G$ is primitive on $B$, then $\core_P(Q)=1$ and $\Inn(T)$ is a subgroup of  $\varphi(Q)$. 
\end{lemma}

\begin{proof}
By~\cite[Theorem~4.7B]{DixMor1996}, the action of $P$ by conjugation on the $k$ simple factors of $B$ is faithful, so $P$ is faithful on $\bm{k}$ by~(\ref{eqn:top}). In other words,  $\core_P(Q)=1$. Taking $M$ to be $P$ in Theorem \ref{thm:twmax}, we see that  $T\simeq V/U\simeq \varphi(V)\leq \Inn(T)$, so $\Inn(T)=\varphi(V)\leq  \varphi(Q)$. 
\end{proof}

Primitive permutation groups of twisted wreath type are therefore  quite rare. For example, Lemma~\ref{lemma:twmax} implies that $T$ is a section of the symmetric group $S_k$, so there are only finitely many primitive groups of twisted wreath type for each $k$. 

Assuming only the necessary conditions for primitivity that are given by Lemma~\ref{lemma:twmax}, we now determine the possibilities for $T$ when $P$ is  $S_k$ or  $A_k$. 

\begin{lemma}
\label{lemma:twSkAk}
Assume Hypothesis~\emph{\ref{hyp:tw}}. If $ \Inn(T)\leq \varphi(Q)$ and $P$ is $S_k$ or $A_k$, then $T\simeq A_{k-1}$ and $k\geq 6$.
\end{lemma}

\begin{proof}
Observe that  $Q$ is  $S_{k-1}$ or $A_{k-1}$. Since $\varphi(Q)$ is not soluble and $\ker(\varphi)\unlhd Q$, it follows that $\varphi$ is injective and  $k\geq 6$. Since $\Inn(T)\unlhd \varphi(Q)$, we conclude that $T\simeq A_{k-1}$.
\end{proof}

Note that for any $k\geq 6$, there is a primitive permutation group of twisted wreath type with socle $A_{k-1}^k$ and top group $S_k$ or $A_k$ (see Example~\ref{exp:AS}). 

The following basic result provides a useful method for constructing examples of quasiprimitive and primitive  groups of twisted wreath type.

\begin{lemma}
\label{lemma:QPprod}
Let $H$ be a quasiprimitive permutation group on $\Delta$ of twisted wreath type with socle $A\simeq T^m$, where $T$ is a simple group. Let $K$ be a transitive subgroup of $S_r$ where $r\geq 2$. Then $G:=H\wr K$ is a quasiprimitive permutation group on $\Omega:=\Delta^r$ of twisted wreath type with socle $B:=A^r$. Further,  $G$ is primitive on $\Omega$ if and only if $H$ is primitive on $\Delta$.
\end{lemma}

\begin{proof}
The socle $A$ of $H$ is non-abelian, non-simple, regular and the unique minimal normal subgroup of $H$. Thus   $B$ is regular on $\Omega$. Further, $C_H(A)=1$, so $C_G(B)=1$.  Since $H$ acts transitively by conjugation on the simple factors of $A$ and $K$ is transitive on $\bm{r}$, it follows that $G$ acts transitively by conjugation on the simple factors of $B$, so  
 $B$ is a minimal normal subgroup of $G$. Thus $G$ is a quasiprimitive permutation group  of twisted wreath type with socle $B$. Lastly, $G$ is primitive if and only if $H$ is primitive by~\cite[Lemma~2.7A]{DixMor1996}.
\end{proof}

In fact, by a similar proof, if $G$ is any large subgroup of $H\wr S_r$ that contains $B$ (see~\S\ref{s:prelim}), then $G$ is a quasiprimitive permutation group of twisted wreath type with socle $B$. However, $G$ may not be primitive on $\Omega$ when $H$ is primitive on $\Delta$.  

Recall the definition of a blow-up from~\S\ref{s:prelim}.  As defined in~\cite{Bad1993}, a  primitive permutation group of twisted wreath type is said to be of \textit{minimal-twisted type} if it is not a blow-up of index $r\geq 2$ of any primitive permutation group of twisted wreath type. Baddeley gives a detailed analysis of these groups in~\cite[\S8--9]{Bad1993}. We state an explicit version of his theory that is suitable for our purposes. In fact, we give more detail than is necessary to highlight its potential for the further study of   base sizes of twisted wreath products (see Remark~\ref{remark:mintwisted}). 
First we need some  notation. Under Hypothesis~\ref{hyp:tw}, a subgroup $S$ of $G$ is \textit{balanced} if  $Q\leq  S\leq P$ and $\core_S(U)=\core_S(V)$. Let  $\mathcal{S}(G)$ denote the set of balanced subgroups of $G$.  Observe that $\mathcal{S}(G)$ is a partially ordered set under inclusion, and $P\in\mathcal{S}(G)$. Note that if $G$ is primitive on $B$, then by Lemma~\ref{lemma:twmax},  $Q\notin\mathcal{S}(G)$. The following result is a weaker version of~\cite[Theorems~8.3 and~9.12]{Bad1993}.
  
\begin{thm}
\label{thm:blowup}
Assume Hypothesis~\emph{\ref{hyp:tw}} where $G$ is primitive on $B$, so $\core_P(Q)=1$. Then one of the following holds.
\begin{itemize}
\item[(i)] $\mathcal{S}(G)=\{P\}$, and one of the following holds.  
\begin{itemize}
\item[(a)] $P$ is almost simple and quasiprimitive on $\bm{k}$. 
\item[(b)]  $P$ is primitive of diagonal type on $\bm{k}$, and $P$ has a unique minimal normal subgroup that is isomorphic to $T^\ell$ for some $\ell$.
\end{itemize}
\item[(ii)] $\mathcal{S}(G)\neq \{P\}$. Let $S$ be minimal in $\mathcal{S}(G)$. Let $\overline{P}:=S/\core_S(Q)$, $\overline{Q}:=Q/\core_S(Q)$ and $\overline{\varphi}:\overline{Q}\to \Aut(T)$ be the homomorphism induced by $\varphi$. 
Then $\overline{G}:= T\twr_{\overline{\varphi}} \overline{P}$ is primitive on  $\overline{B}:=\soc(\overline{G})$, and $\mathcal{S}(\overline{G})=\{\overline{P}\}$. 
Further, up to permutation isomorphism,
$$
\overline{B}^r\leq  G\leq \overline{G}\wr S_r,
$$
where $r:=[P:S]$, the wreath product  $\overline{G}\wr S_r$ acts  via the product action on   $\overline{B}^r$, and  the image of the projection map from $G$ to $S_r$ is transitive on $\bm{r}$.
\end{itemize}
\end{thm}

\begin{proof}
Suppose that $\mathcal{S}(G)=\{P\}$. Then $G$ is of minimal-twisted type by~\cite[Theorem~8.3(3)]{Bad1993}.  By~\cite[Theorem~9.12]{Bad1993}, $G$ is permutation isomorphic to a twisted wreath product 
$G'=T'\twr_{\varphi'} P'$  described  by~\cite[Constructions~9.1 or~9.3]{Bad1993}. In particular, the point stabiliser $Q'$ is core-free in $P'$, so $P'$ is permutation isomorphic to $P$. Thus either 
$P$ is  primitive  of diagonal type with a unique minimal normal subgroup  $T^\ell$ for some $\ell$, in which case (i)(b) holds, or $P$ is almost simple, in which case $P=\soc(P)Q$ by~\cite[Proposition~9.2(2)]{Bad1993}, so $\soc(P)$ is transitive and (i)(a) holds.

Thus we may assume that $S\in \mathcal{S}(G)$ is minimal, as in  (ii).  By~\cite[Proposition~8.4]{Bad1993},  $N:=\core_S(Q)=\core_S(V)=\core_S(U)$.
For any subset $X$ of $S$, let $\overline{X}:=\{Nx : x \in X\}$, and let $\overline{P}:=\overline{S}$, as in (ii). Now $\core_{\overline{P}}(\overline{Q})=1$, so we may define $\overline{G}$ as in (ii). 
Note that $\overline{V}=\overline{\varphi}^{-1}(\Inn(T))$ and $\overline{U}=\ker(\overline{\varphi})$. If $\overline{R}\in \mathcal{S}(\overline{G})$, then $\overline{Q}\leq \overline{R}\leq \overline{S}$ and $\core_{\overline{R}}(\overline{U})=\core_{\overline{R}}(\overline{V})$, so  $R\in \mathcal{S}(G)$ and $R\leq S$. Thus $R=S$, so $\mathcal{S}(\overline{G})=\{\overline{P}\}$.

Since $S$ contains the stabiliser $Q=P_1$, it follows that $S=P_I$ where $I$ is the block $1^S$ of $P$. In particular, the set  
$A:=\{a_i:i\in I\}$ is a left transversal for $Q$ in $S$ with the property that $i^{a_i}=1$ for $i\in I$ (see Hypothesis~\ref{hyp:tw}). Further, $\overline{A}$ is a left transversal for $\overline{Q}$ in $\overline{P}$. Thus 
$$\overline{B}\simeq T^{\overline{A}}\simeq T^A\simeq \{f\in B : f(x)=1\ \mbox{for}\ x\in P\setminus S\}=:B(S).$$ 
Note that by~(\ref{eqn:tw}),  $C_{S}(B(S))=N$  and $G_{B(S)}=B(S){:}S$. It is then routine to verify that $\overline{G}=\overline{B}{:}\overline{P}$ is permutation isomorphic to $G_{B(S)}/N$ in its  action on $B(S)$.  By the proof of~\cite[Theorem~8.3(2)]{Bad1993},  $G$ is a blow-up of $G_{B(S)}/N$ of index $r$, so $G$ is a blow-up of $\overline{G}$ of index $r$.
Thus $\overline{G}$ is primitive by Lemma~\ref{lemma:QPprod}, and  (ii) holds. 
\end{proof}

\begin{remark}
We make several observations concerning Theorem~\ref{thm:blowup}.
\begin{enumerate}
\item By Theorem~\ref{thm:blowup}, its proof and  Theorem~\ref{thm:twONan}, a primitive permutation group $G$  of twisted wreath type with top group $P$ is of minimal-twisted type if and only if $\mathcal{S}(G)=\{P\}$. 
\item If the top group of a primitive group $G$ of twisted wreath type is primitive, then certainly $G$ is of minimal-twisted type, but the converse is not true (see~\cite[Example~9.8]{Bad1993}). 
\item The groups $P$ described in (i)(b) are precisely the primitive groups of type III(a)(i) in Table~\ref{tab:ONS}. In fact, a primitive twisted wreath product can be constructed using any quasiprimitive $P$ in class III(a)(i) 
(see~Example~\ref{exp:diag}), and  every 
primitive twisted wreath product satisfying (i)(b) arises in this way (see the proof of Theorem~\ref{thm:blowup} and \cite[Construction~9.3]{Bad1993}). Thus the minimal-twisted type  groups whose top groups are primitive of diagonal type are well understood. 
We also conclude that there are  primitive twisted wreath products  that are not of minimal-twisted type but have quasiprimitive top groups.
\item 
Much less is known about the primitive twisted wreath products satisfying (i)(a); see~\cite[Construction~9.1 and \S10]{Bad1993}. In Example~\ref{exp:AS}, we show that for any non-abelian simple group $T$, there is a minimal-twisted type group with socle $T^k$ for some $k$ whose top group is almost simple and primitive. See~\cite[\S9]{Bad1993} for more examples. 
\item The set $\mathcal{S}(G)$ may contain more than one minimal element; see~\cite[Example~9.11]{Bad1993}. 
\end{enumerate}
\end{remark}

\section{Probabilistic results}
\label{s:prob}

In this section, we prove the following result, and then we use this result to prove Theorem~\ref{thm:prob}. Some of the methods in this section will also be used to prove Theorem~\ref{thm:QP}. 

\begin{thm}
\label{thm:primprobgeneral}
Suppose that Hypothesis~\emph{\ref{hyp:tw}} holds  and   $\core_P(Q)=1$. If $P$ is primitive on $\bm{k}$, and if  $\Inn(T)\leq \varphi(Q)$ when $P$ is $S_k$ or $A_k$, then the proportion of pairs of points from $B$ that are bases for $G$ tends to 1 as $|G|\to \infty$.
\end{thm}
 
Note that the conclusion of  Theorem~\ref{thm:primprobgeneral} does not hold if we remove the assumption  that $\Inn(T)\leq \varphi(Q)$ when $P$ is $S_k$ or $A_k$; see Example~\ref{exp:Sktop}. 

First we need some notation. For a  transitive permutation group $G$ on $\Omega$ and an integer $b\geq 1$, let $Q(G,b)$ denote the proportion of $b$-tuples in $\Omega^b$ that are not (ordered) bases for $G$. The following observation was made by Liebeck and Shalev~\cite[p.502]{LieSha1999}.

\begin{lemma}
\label{lemma:prob}
Let $G$ be a  transitive permutation group on $\Omega$. For any integer $b\geq 1$ and $\alpha\in \Omega$,
$$ Q(G,b)\leq \sum_{x\in R} \frac{| x^G \cap G_\alpha|^b|C_G(x)|^{b-1}}{|G|^{b-1}},$$
 where $R$ is any set of representatives  for the $G$-conjugacy classes of prime order elements  in~$G_\alpha$.
\end{lemma}

In order to apply Lemma~\ref{lemma:prob} in the case where $G$ is a twisted wreath product, we require two basic results concerning conjugacy classes and centralisers of elements of $G$.

\begin{lemma}
\label{lemma:conjcent}
Under Hypothesis~\emph{\ref{hyp:tw}}, the following hold for $x\in P$.
\begin{itemize}
\item[(i)] $x^G\cap P=x^P$.
\item[(ii)] $C_G(x)=\{fy\in G: f\in \fix_B(x), y\in C_P(x) \}$.
\item[(iii)] $|C_G(x)|=|\fix_B(x)||C_P(x)|$. 
\end{itemize}
\end{lemma}

\begin{proof}
(i) Let  $y\in x^G\cap P$. Then 
$y=(fz)^{-1}x(fz)$ for some $f\in B$ and $z\in P$. Now  $(fz)^{-1}x(fz)=(f^{-1} f^{x^{-1}})^z z^{-1}xz$. Since $G=B{:}P$ and $y\in P$, we must have  $y=z^{-1}xz\in x^P$.

(ii) Let $f\in B$ and $y\in P$. Then $fy$ centralises $x$ if and only if $fyx=f^{x^{-1}}xy$, and this occurs precisely when $f\in \fix_B(x)$ and $y\in C_P(x)$.

(iii) This follows from (ii) since $G=B{:}P$.
\end{proof}

Recall from \S\ref{s:prelim} that for $x\in P$ where $P$ acts on $\bm{k}$, we denote the number of orbits of $\langle x\rangle$ on $\bm{k}$ by $\omega_{\bm{k}}(x)$.

\begin{lemma}
\label{lemma:fix}
Under Hypothesis~\emph{\ref{hyp:tw}}, $|\fix_B(x)|\leq |T|^{\omega_{\bm{k}}(x)}$ for all $x\in P$.
\end{lemma}

\begin{proof}
Let $x\in P$. Without loss of generality, we may assume that the  $r:=\omega_{\bm{k}}(x)$  orbits of $\langle x\rangle$ on $\bm{k}$ have representatives  $1,\ldots,r$. 
Recall that  $\{a_1,\ldots,a_k\}$ is a left transversal for $Q$ in $P$. 
Define a map $\psi:\fix_B(x)\to T^r$ by $f\mapsto (f(a_1),\ldots,f(a_r))$. We claim that $\psi$ is injective. To this end, suppose that  $(f_1(a_1),\ldots,f_1(a_r))=(f_2(a_1),\ldots,f_2(a_r))$ for some $f_1,f_2\in \fix_B(x)$. Let $i\in \bm{k}$. Now  $i=j^y$ for some $j\in \bm{r}$ and $y\in \langle x\rangle$. By Lemma \ref{lemma:ideal},  $f(a_j)^{\varphi(q_{y,j})}=f(a_{j^y})$ for $f\in \{f_1,f_2\}$, so $$f_1(a_i)=f_1(a_{j^y})=f_1(a_j)^{\varphi(q_{y,j})}=f_2(a_j)^{\varphi(q_{y,j})}=f_2(a_{j^y})=f_2(a_i).$$ 
Thus $f_1=f_2$, and the claim holds. 
\end{proof}

Next we obtain a useful version of Lemma~\ref{lemma:prob} for twisted wreath products. 

\begin{lemma}
 \label{lemma:form}
 Assume Hypothesis~\emph{\ref{hyp:tw}}. 
For any integer $b\geq 1$, 
$$ Q(G,b)\leq \sum_{x\in R(P)}\frac{ |x^P|}{ |T|^{(b-1)(k-\omega_{\bm{k}}(x))}}, 
$$
where $R(P)$ is any set of representatives for the conjugacy classes of prime order  elements  in $P$.
\end{lemma}

\begin{proof}
 By Lemma \ref{lemma:conjcent}(i), $R(P)$ is a  set of representatives for the $G$-conjugacy classes of prime order elements in $G_e=P$, and $|x^G\cap G_e|=|x^P|$ for all $x\in R(P)$. Moreover, Lemma  \ref{lemma:conjcent}(iii) and Lemma \ref{lemma:fix} imply that $|C_G(x)|=|\fix_B(x)||C_P(x)|\leq |T|^{\omega_{\bm{k}}(x)}|C_P(x)|$ for all $x\in R(P)$. Since $|G|=|T|^k|P|$, Lemma \ref{lemma:prob} implies that
$$Q(G,b)\leq  \sum_{x\in R(P)} \frac{|x^P|^b(|T|^{\omega_{\bm{k}}(x)}|C_P(x)|)^{b-1}}{(|T|^k|x^P||C_P(x)|)^{b-1}},$$
 and the result follows.
\end{proof}

\begin{proof}[Proof of Theorem~\emph{\ref{thm:primprobgeneral}}]
By assumption, $P$ is a primitive subgroup of $S_k$, and $\Inn(T)\leq \varphi(Q)$ when $P$ is $S_k$ or $A_k$. 
It suffices to prove that $Q(G,2)\to 0$ as $|G|\to\infty$.
Let $R(P)$ be a set of representatives for the conjugacy classes of prime order elements in $P$. By Lemma \ref{lemma:form}, $Q(G,2)\leq \sum_{x\in R(P)} |x^P|/ |T|^{k-\omega_{\bm{k}}(x)}$.  We claim that 
$$\sum_{x\in R(P)}\frac{ |x^P|}{ |T|^{k-\omega_{\bm{k}}(x)}}\leq \frac{C}{|T|^{\frac{1}{3}}}\left( \frac{1}{c^k}+\frac{1}{\sqrt{k}} + \frac{k^2+4 k}{ (k-1)!^{\frac{2}{3}}} \right) $$
for some absolute constants $C$ and $c>1$, in which case  $Q(G,2)\to 0$ as $|G|\to\infty$ since the fact that $|G|\leq |T|^kk!$ forces $|T|\to\infty$ or $k\to \infty$ as $|G|\to\infty$.

If $P$ is not $S_k$ or $A_k$, then the claim holds by~\cite[Lemma~4.4]{Faw2013}, so we assume that $P$ is $S_k$ or $A_k$. Then $T=A_{k-1}$ by  Lemma \ref{lemma:twSkAk}. Note that  $k-\omega_{\bm{k}}(x)$ is 1 when $x$ is a transposition and at least 2 otherwise.  Then 
 $$ 
\sum_{x\in R(P)}\frac{ |x^P|}{ |T|^{k-\omega_{\bm{k}}(x)-\frac{1}{3}}}
\leq  \frac{ |(12)^{S_k}|}{ |T|^{1-\frac{1}{3}}} + \frac{|S_k|}{ |T|^{2-\frac{1}{3}}}
= \frac{ k(k-1)}{ (k-1)!^{\frac{2}{3}}2^{\frac{1}{3}}}+\frac{2^{\frac{5}{3}} k}{ (k-1)!^{\frac{2}{3}}}
\leq \frac{k^2+4k}{ (k-1)!^{\frac{2}{3}}},
$$
and the claim follows.
\end{proof}

\begin{remark}
It would be interesting to determine whether Theorem~\ref{thm:primprobgeneral} still holds if we weaken the assumption that $P$ is  primitive to quasiprimitive. By Theorem~\ref{thm:QPgeneral}, such twisted wreath products have base size $2$, so this is a reasonable problem to consider. It would suffice to  adapt the proof of \cite[Lemma 4.4]{Faw2013}. The key  ingredients of this proof are the following  properties of primitive subgroups  $P$ of $S_k$:  (a) if $P$ contains a transposition, then $P=S_k$; (b) either $A_k\leq P$, or $|P|\leq 4^k$~\cite{PraSax1980}; (c) either $A_k\leq P$, or $|P|\leq \exp(4\sqrt{k}\log^2{k})$ for sufficiently large $k$~\cite{Bab1982}; (d) either the minimal degree of $P$ is at least $k/3$, or $A_m^r\leq P\leq S_m\wr S_r$ and $P$ acts via the product action  on $\Delta^r$ where $\Delta$ is the set of $\ell$-subsets of $\bm{m}$ and $m\geq 5$, $r\geq 1$ and $1\leq\ell<m/2$~\cite{LieSax1991}. 
By~\cite{PraSha2001}, (a), (b) and (c)  hold for any quasiprimitive subgroup $P$ of $S_k$.
\end{remark}

\begin{proof}[Proof of Theorem~\emph{\ref{thm:prob}}]
Let $G$ be a primitive permutation group of twisted wreath type whose top group is primitive.  By Theorem~\ref{thm:twONan}, we may assume that $G$ is given by  Hypothesis~\ref{hyp:tw}, so $G$ is primitive on $B$ and the top group  $P^{\bm{k}}$ is primitive on $\bm{k}$.  By Lemma~\ref{lemma:twmax}, $\core_P(Q)=1$ and $\Inn(T)\leq \varphi(Q)$, so the result follows from Theorem~\ref{thm:primprobgeneral}. 
\end{proof}

\section{Bounds on base sizes}
\label{s:bounds}

In this section, we establish two upper bounds on the base size of a twisted wreath product. 
 The proof of the first bound is elementary, but we  use the CFSG in the second to deduce that $T$ is $2$-generated. Recall that if Hypothesis~\ref{hyp:tw} holds, then  $\varphi(Q)$ is a subgroup of $\Aut(T)$, and for any subgroup $H$ of $\Aut(T)$, we denote the number of orbits of $H$  on $T$ by $\omega_T(H)$. Recall also that $d_{\bm{k}}(P)$ denotes the distinguishing number of $P$ on $\bm{k}$ (see~\S\ref{s:dist}).

\begin{lemma}
\label{lemma:hbound} 
If Hypothesis~\emph{\ref{hyp:tw}} holds and $\core_P(Q)=1$, then 
$$b_{B}(G)\leq \left \lceil\frac{\log{d_{\bm{k}}(P)}}{\log{\omega_{T}(\varphi(Q))}}\right\rceil +1\leq \left \lceil\frac{\log{d_{\bm{k}}(P)}}{\log{\omega_{T}(\Aut(T))}}\right\rceil +1.$$
\end{lemma}

\begin{proof}
Let $d:=d_{\bm{k}}(P)$, $n:=\omega_{T}(\varphi(Q))$ and $m:=\left \lceil\log{d}/\log{n}\right\rceil$. Note that $d\geq 2$, so $m\geq 1$. Note also that $n^{m-1}\leq d-1<n^m$. For each integer $u$ such that $ 0\leq u\leq d-1$, let $d_{u,0},\ldots, d_{u,m-1}$ denote the first $m$ digits of the base $n$ representation of $u$, so 
$$u= d_{u,0}+d_{u,1}n+\cdots +d_{u,m-1}n^{m-1}.$$
Let $t_0,\ldots,t_{n-1}$ be representatives of the orbits of $\varphi(Q)$ on $T$, and let $\{\Delta_0,\ldots,\Delta_{d-1}\}$ be a distinguishing partition for $P$ on $\bm{k}$. Recall from Hypothesis~\ref{hyp:tw} that 
$L=\{a_1,\ldots,a_k\}$ is a left transversal for $Q$ in $P$, and 
recall from \S\ref{s:tw} that to define an element of the base group $B$, it suffices to define an element of  $T^L$. For $1\leq j\leq m$, define   $f_j(a_i):=t_{d_{u,j-1}}$ whenever $i\in \Delta_u$. Then  $\mathcal{B}:=\{f_1,\ldots,f_m\}\subseteq B$. Suppose that $x\in P$ fixes $\mathcal{B}$ pointwise. Let $i\in \Delta_u$ where $0\leq u\leq d-1$. Then   $i^x\in \Delta_v$ for some $0\leq v\leq d-1$. By Lemma \ref{lemma:ideal},   $f_j(a_i)^{\varphi(q_{x,i})}=f_j(a_{i^x})$ for  $1\leq j\leq m$, so $t_{d_{u,j-1}}$ and $t_{d_{v,j-1}}$ are in the same orbit of  $\varphi(Q)$ for  $1\leq j\leq m$. Thus   $d_{u,j-1}=d_{v,j-1}$ for $1\leq j\leq m$, so $u=v$. We have proved that $\Delta_u^x=\Delta_u$ for $0\leq u\leq d-1$, so $x\in \core_P(Q)=1$. Hence $\mathcal{B}$  is a base for the action of $P$ on $B$, and we are done by  Lemma~\ref{lemma:PonB}.
\end{proof}

\begin{lemma}
\label{lemma:Tbound} 
If Hypothesis~\emph{\ref{hyp:tw}} holds, then 
$$b_{B}(G)\leq \left \lceil\frac{\log{d_{\bm{k}}(P)}}{\log{|T|}}\right\rceil+3.$$
\end{lemma}

\begin{proof}
Let  $d:=d_{\bm{k}}(P)$, $n:=|T|$  and $m:=\left \lceil\log{d}/\log{n}\right\rceil$. For each integer $u$ such that $ 0\leq u\leq d-1$, let $d_{u,0},\ldots, d_{u,m-1}$ denote the first $m$ digits of the base $n$ representation of $u$, so 
$$u= d_{u,0}+d_{u,1}n+\cdots +d_{u,m-1}n^{m-1}.$$
 By \cite{AscGur1984}, there exist $t,s\in T$ such that $T=\langle t,s\rangle$. Enumerate the elements of $T$ as $t_0,\ldots,t_{|T|-1}$, and let $\{\Delta_0,\ldots,\Delta_{d-1}\}$ be a distinguishing partition for $P$ on $\bm{k}$. For $1\leq j\leq m$, let   $f_j(a_i):=t_{d_{u,j-1}}$ whenever $i\in \Delta_u$. Let $f_{m+1}(a_i):=t$ and $f_{m+2}(a_i):=s$ for $i\in \bm{k}$. 
  Then  $\mathcal{B}:=\{f_1,\ldots,f_{m+2}\}\subseteq B$. Suppose that $x\in P$ fixes $\mathcal{B}$ pointwise. 
By Lemma~\ref{lemma:ideal}, $f(a_i)^{\varphi(q_{x,i})}=f(a_{i^x})$ for  all  $i\in\bm{k}$ and $f\in \mathcal{B}$.  
 In particular, for each $i\in\bm{k}$,  the automorphism $\varphi(q_{x,i})$ fixes $t$ and $s$, so $\varphi(q_{x,i})=1$.  Let $i\in \Delta_u$ where $0\leq u\leq d-1$. Then   $i^x\in \Delta_v$ for some $0\leq v\leq d-1$. 
Now $t_{d_{u,j-1}}=f_j(a_i)=f_j(a_{i^x})=t_{d_{v,j-1}}$ for $1\leq j\leq m$. Thus   $d_{u,j-1}=d_{v,j-1}$ for $1\leq j\leq m$, so $u=v$.  We have proved that $\Delta_u^x=\Delta_u$ for $0\leq u\leq d-1$, so $x\in \core_P(Q)$. For  $i\in\bm{k}$, note that  $q_{x,i}\in U$ since $\varphi(q_{x,i})=1$, and 
observe that  $q_{x,i}=a_i^{-1}xa_{i^x}=a_i^{-1}xa_i$.  Thus 
$x\in \bigcap_{i\in \bm{k}} a_iUa_i^{-1}=\core_P(U)\leq \core_P(V)=1$. 
 Hence $\mathcal{B}$  is a base for  $P$, and we are done by Lemma~\ref{lemma:PonB}.
\end{proof}

Next we use Theorem~\ref{thm:dist} to show that there is a close relationship between the upper bound on $b_B(G)$ in Lemma~\ref{lemma:Tbound} and the  trivial lower bound $\log_{|B|}|G|\leq b_B(G)$.

\begin{lemma}
\label{lemma:keybound}
Assume Hypothesis~\emph{\ref{hyp:tw}}. Then 
\begin{equation}
\label{eqn:key}
\frac{\log{d_{\bm{k}}(P)}}{\log{|T|}}< \frac{\log{|G|}}{\log{|B|}}.
\end{equation}
Further, if $\core_P(Q)=1$, then
\begin{equation}
\label{eqn:primkey}
\frac{\log{|G|}}{\log{|B|}}<  \frac{\log{d_{\bm{k}}(P)}}{\log{|T|}}+1.
\end{equation}
\end{lemma}

\begin{proof}
Let $t:=|T|$. Recall that $|G|=|B||P|$ and $|B|=t^k$. Let $R:=P^{\bm{k}}$, so $d:=d_{\bm{k}}(P)=d_{\bm{k}}(R)$.
By Theorem~\ref{thm:dist}, $d\leq 48\sqrt[k]{|R|}\leq 48\sqrt[k]{|P|}$, and clearly $48< t$, so
$$
\log_t d\leq \log_t(48\sqrt[k]{|P|})
< 1+\log_t\sqrt[k]{|P|}=1+\log_{t^k}|P|=\log_{|B|}|G|,
$$
in which case~(\ref{eqn:key}) holds. Now suppose that $\core_P(Q)=1$. Then $\sqrt[k]{|P|}< d$ (see~\S\ref{s:dist}), in which case~(\ref{eqn:primkey}) holds by the observations above.
\end{proof}

We finish this section with a useful consequence of Lemma~\ref{lemma:hbound}.

\begin{lemma}
\label{lemma:hboundplus}
Let $G$ be a primitive permutation group of twisted wreath type with socle $T^k$ and top group $R$, where $T$ is a simple group. If $d_{\bm{k}}(R)\leq \omega_{T}(\Aut(T))$, then $b(G)=2$.
\end{lemma}

\begin{proof}
By Theorem~\ref{thm:twONan}, we may assume that $G$ is given by Hypothesis~\ref{hyp:tw} and $R=P^{\bm{k}}$. Now $G$ is primitive on $B$, so $\core_P(Q)=1$ by Lemma~\ref{lemma:twmax},  and the result follows from Lemma~\ref{lemma:hbound}.
\end{proof}

\section{Base size $2$}
\label{s:QP}

In this section, we prove the following result, which shows that a large class of quasiprimitive permutation groups of twisted wreath type have base size $2$ (see Remark~\ref{remark:base2}). We then use this result to prove Theorem~\ref{thm:QP}. Recall that a permutation group is semiprimitive if every normal subgroup is transitive or semiregular.

\begin{thm}
\label{thm:QPgeneral}
Suppose that Hypothesis~\emph{\ref{hyp:tw}} holds  and   $\core_P(Q)=1$. If any of the following hold, then $b_B(G)=2$.
\begin{itemize}
\item[(i)] $P$ is semiprimitive on~$\bm{k}$, and if $P$ is $S_k$ or $A_k$, then  $\Inn(T)\leq \varphi(Q)$.
\item[(ii)] $A_4$ is not a section of $P$.  
\item[(iii)] The order of $P$ is not divisible by $4$.
\item[(iv)]  $A_5$ is not a section of $P$ and $T\not\simeq  A_5$.
\item[(v)] $P$ is soluble and $T\not\simeq  A_5$.
\item[(vi)]   $A_s$ is not a section of $P$ and $3\leq s\leq \omega_{T}(\Aut(T))$.
\end{itemize}
\end{thm}

\begin{remark}
\label{remark:base2}
We make several observations concerning Theorem~\ref{thm:QPgeneral}.
\begin{enumerate}
\item Many quasiprimitive  permutation groups of twisted wreath type satisfy the conditions of Theorem~\ref{thm:QPgeneral}
 and therefore have base size $2$. In fact, by Remark~\ref{remark:lotsQP}, these conditions can be satisfied  for any non-abelian simple group $T$ and transitive subgroup $P$ of $S_k$ such that either $P$ is semiprimitive on $\bm{k}$ but not $S_k$ or $A_k$, or $T$ and $P$ satisfy one of (ii)--(vi) where $k\geq 2$.  Similarly, if $P=S_k$ or $A_k$  where $k\geq 6$, then the conditions of Theorem~\ref{thm:QPgeneral} can be  satisfied with $T=A_{k-1}$ by Example~\ref{exp:AS}. 
\item There are conditions in  Theorem~\ref{thm:QPgeneral}   that cannot hold for any    primitive permutation group of twisted wreath type.
Indeed, if Hypothesis~\ref{hyp:tw} holds and $G$ is primitive on $B$, then $T$ is a section of $P$ by Lemma~\ref{lemma:twmax}, so (iii) and  (v) cannot hold.
Moreover, if (i) holds, then $P$  is  quasiprimitive. This is because any semiregular normal subgroup of $P$ is trivial: if $S$ is such a subgroup, then $S\cap Q=1$,
 so $(SU)\cap V=U$, but then we may take  $M=P$ and $R=SU$ in Theorem~\ref{thm:primitivetwr}(iii), 
 so $R=U$ and $S=1$. (The author thanks Luke Morgan for providing this argument.) 
\item  The assumption that $\Inn(T)\leq \varphi(Q)$ when $P$ is $S_k$ or $A_k$ is required; see Example~\ref{exp:Sktop}.
\end{enumerate}
\end{remark}

In order to prove Theorem~\ref{thm:QPgeneral}, we first  consider the case where $P$ is $S_k$ or $A_k$.

\begin{lemma}
 \label{lemma:QPSkAk}
Assume Hypothesis~\emph{\ref{hyp:tw}}. If $ \Inn(T)\leq \varphi(Q)$ and $P$ is $S_k$ or $A_k$, then $b_B(G)=2$.
\end{lemma}

\begin{proof}
By Lemma \ref{lemma:twSkAk}, $T=A_{k-1}$ and $k\geq 6$. First suppose that $k\geq 8$.  Let $m:=k-1$. Now $d_{\bm{k}}(P)\leq m+1$, so by Lemma \ref{lemma:hbound}, it suffices to show that $m< \omega_{A_m}(S_m)$.  Observe that  $ \omega_{A_m}(S_m)$ is the number of cycle types of even permutations of $\bm{m}$.   Write $m+1=2^ir$ where $i$ is a non-negative integer  and $r$ is an  odd integer. Then $A_{m+1}$ contains a fixed-point-free permutation consisting of $2^i$ disjoint $r$-cycles unless $r=1$, in which case $A_{m+1}$ contains a fixed-point-free permutation consisting of $2^{i-1}$ disjoint transpositions. Thus $\omega_{A_m}(S_m)+1\leq \omega_{A_{m+1}}(S_{m+1})$. Since $\omega_{A_7}(S_7)=8$, the  result follows by induction. 

It remains to consider the case where $k=6$ or $7$. Let $R(P)$ be a set of representatives for the conjugacy classes of prime order elements in $P$. By Lemma~\ref{lemma:form}, it suffices to prove that $\sum_{x\in R(P)} |x^P||T|^{\omega_{\bm{k}}(x)}<|T|^k$, for then $Q(G,2)<1$, so $b_B(G)=2$. This is routine.
\end{proof}

\begin{proof}[Proof of Theorem~\emph{\ref{thm:QPgeneral}}]
By Lemma~\ref{lemma:hbound}, $b_B(G)=2$  if $d_{\bm{k}}(P)\leq \omega_T(\Aut(T))$. By Theorem~\ref{thm:summary}(iii), if $A_s$ is not a section of $P$, then $d_{\bm{k}}(P)\leq s$. Hence if  (vi) holds, then $b_B(G)=2$.

Note that $\omega_T(\Aut(T))\geq 4$ since  $3$ distinct primes divide  $|T|$ by  Burnside's $p^aq^b$ Theorem  \cite[Theorem 31.4]{JamLie2001}. Thus $b_B(G)=2$ if $d_{\bm{k}}(P)\leq 4$. In particular, if (ii) or (iii) holds, then (ii) holds, so $b_B(G)=2$. Moreover, if (i) holds, then $b_B(G)=2$ by Theorem~\ref{thm:summary}(i) and Lemma~\ref{lemma:QPSkAk}.

By \cite[Theorem 2.3]{Str2001}, $\omega_T(\Aut(T))=4$  if and only if $T\simeq A_5$.   Thus $b_B(G)=2$ if $d_{\bm{k}}(P)\leq 5$ and $T\not\simeq A_5$. In particular, if (iv) or (v) holds, then (iv) holds, so $b_B(G)=2$.
\end{proof}

\begin{proof}[Proof of Theorem~\emph{\ref{thm:QP}}]
By assumption, $G$ is a primitive permutation group of twisted wreath type whose top group is quasiprimitive. By Theorem~\ref{thm:twONan}, we may assume that $G$ is given by  Hypothesis~\ref{hyp:tw}, so $G$ is primitive on $B$, and $P^{\bm{k}}$ is quasiprimitive. By Lemma~\ref{lemma:twmax}, $\core_P(Q)=1$ and $\Inn(T)\leq \varphi(Q)$, so the result follows from Theorem~\ref{thm:QPgeneral}(i) (since any quasiprimitive permutation group is semiprimitive by definition).  
\end{proof}

\section{Results related to Pyber's conjecture}
\label{s:pyber}

In this section, we prove Theorems~\ref{thm:pyber} and~\ref{thm:exp} using the following result.

\begin{thm}
\label{thm:expgeneral}
Let $H$ be a permutation group on $\Delta$ with a regular subgroup  $A\simeq T^m$ where $T$ is  a group, $m\geq 1$ and $|A|\geq 48$. Let $\Omega:=\Delta^r$ where  $r\geq 2$. Let $A^r\leq G\leq H\wr S_r$ where the image $K$ of the projection map $\rho:G\to S_r$ is transitive on $\bm{r}$. 
If  $b_\Delta(H)=2$, then 
$$
b_\Omega(G)=\left\lceil \frac{\log|G|}{\log|\Omega|}\right\rceil+\varepsilon=\left\lceil \frac{\log{d_{\bm{r}}(K)}}{m\log {|T|}}\right\rceil+\delta
$$
for some $\varepsilon,\delta\in \{0,1,2\}$ where $\varepsilon\leq \delta$. Further, if $G=H\wr K$ and $[H:A]\geq 48$, then $\varepsilon\in \{0,1\}$, $\delta\in\{1,2\}$ and $\varepsilon+1\leq \delta$.
\end{thm}

In Theorem~\ref{thm:expgeneral}, we may take $H$  to be any quasiprimitive permutation group of twisted wreath type since the socle $A$ of $H$ is regular. Note, however, that even for such $H$, the group $G$ may not be quasiprimitive. For example, if $K\leq S_r$ is transitive, then we may take $G:=A\wr K$, but $A^r$ is a regular normal subgroup of $G$ that is not minimal.  Nevertheless, by  Lemma~\ref{lemma:QPprod} and Remark~\ref{remark:base2}(1), there are many  quasiprimitive permutation groups $G$ of twisted wreath type that satisfy the conditions of Theorem~\ref{thm:expgeneral}.  

To prove Theorem~\ref{thm:expgeneral}, we require the following elementary result, which we prove by adapting the  proof of  Lemma~\ref{lemma:Tbound}. 

\begin{lemma}
\label{lemma:prod}
Let $H\leq \Sym(\Delta)$ and $K\leq S_r$ where $|\Delta|\geq 2$ and $K\neq 1$. Let $G\leq  H\wr K$ and $\Omega:=\Delta^r$. Then
$$
b_\Omega(G)\leq \left\lceil  \frac{\log d_{\bm{r}}(K)}{\log |\Delta|}\right\rceil+b_\Delta(H).
$$
\end{lemma}

\begin{proof}
Let $d:=d_{\bm{r}}(K)$, $n:=|\Delta|$  and $m:=\left \lceil\log{d}/\log{n}\right\rceil$.  Note that $n^{m-1}\leq d-1<n^m$. We may assume that $\Delta=\{0,\ldots,n-1\}$.
For each integer $u$ such that $ 0\leq u\leq d-1$, let $d_{u,0},\ldots, d_{u,m-1}$ denote the first $m$ digits of the base $n$ representation of $u$, so 
$$u= d_{u,0}+d_{u,1}n+\cdots +d_{u,m-1}n^{m-1}$$
and  $d_{u,0},\ldots, d_{u,m-1}\in \Delta$.  Let $\{\Delta_0,\ldots,\Delta_{d-1}\}$ be a distinguishing partition for $K$ on $\bm{r}$. For $1\leq j\leq m$, let $\alpha_j$ be the element of $\Delta^r$ whose $i$-th coordinate is $d_{u,j-1}$ whenever  $i\in \Delta_u$. Let $b:=b_\Delta(H)$, and let $i_1,\ldots,i_b$ be a base for $H$ on $\Delta$. For $1\leq \ell\leq b$, let $\beta_\ell:=(i_\ell,\ldots,i_\ell)\in \Delta^r$. Let $\mathcal{B}:=\{\alpha_1,\ldots,\alpha_m,\beta_1,\ldots,\beta_b\}$   and suppose that $g=(h_1,\ldots,h_r)\pi\in G$ fixes $\mathcal{B}$ pointwise. Then $h_j$ fixes $i_\ell$ for $1\leq \ell\leq b$ and $1\leq j\leq r$, so $g=(1,\ldots,1)\pi$. Let $i\in \Delta_u$. Now   $i^\pi\in \Delta_v$ for some $v$. For $1\leq j\leq m$, the digit  
$d_{u,j-1}$ is the $i$-th coordinate of $\alpha_j$, so $d_{u,j-1}$ is the $i^\pi$-th coordinate of $\alpha_j^\pi=\alpha_j$, so 
 $d_{u,j-1}=d_{v,j-1}$. Thus $u=v$.  We have shown that $\Delta_u^\pi=\Delta_u$ for $0\leq u\leq d-1$, so $\pi=1$. Hence $\mathcal{B}$ is a base for $G$, and the bound on $b_\Omega(G)$ follows.
\end{proof}

Note that the bound of  Lemma~\ref{lemma:prod} is stated in~\cite[\S4.3.1]{HalLieMar2019} for the case where $H$ is primitive and almost simple and $K$ is transitive.
 They use~\cite[Lemma~2.1]{DuyHalMar2018} to deduce that $m$ is the base size of $K$ in its action on the set of partitions of $\bm{r}$ with at most $|\Delta|$ parts, and there is a natural way of using such a base to construct $\alpha_1,\ldots,\alpha_m$ so that the proof proceeds as above  (see the proof of~\cite[Lemma~3.8]{BurSer2015}).  Thus our  proof is essentially the same but more constructive. 

\begin{proof}[Proof of Theorem~\emph{\ref{thm:expgeneral}}]  
 Let $B:=A^r$. Note that $48\leq |\Delta|=|A|=|T|^m$ and $|\Omega|=|B|=|T|^{mr}$.

Let $c:=1$ when $G=H\wr K$ and $[H:A]\geq 48$, and let $c:=0$ otherwise.
Note that $|B||K|\leq |G|$ since $B\leq \ker(\rho)$, and if $c=1$, then $48^r|B||K|\leq |G|$. 
 By Theorem~\ref{thm:dist},
$$
\frac{\log d_{\bm{r}}(K)}{m\log |T|}+c
\leq \frac{\log(48\sqrt[r]{|K|})}{m\log|T|}+c
=\frac{\log 48}{m\log|T|}+\frac{\log(|B||K|)}{\log |B|}+c-1
\leq \frac{\log|G|}{\log|\Omega|}.
$$
Since $\log_{|\Omega|}|G|\leq b_\Omega(G)$ and  $b_\Delta(H)=2$, the result follows from Lemma~\ref{lemma:prod}.
\end{proof}

\begin{proof}[Proof of Theorem~\emph{\ref{thm:pyber}}]
Recall that $G$ is a quasiprimitive permutation group on $\Omega$ of twisted wreath type with top group $R$. By Theorem~\ref{thm:twONan},  we may assume that $G$ is given by Hypothesis~\ref{hyp:tw}, in which case $\Omega=B$, and we may assume that $R=P^{\bm{k}}$, so $d:=d_{\bm{k}}(R)=d_{\bm{k}}(P)$.
Recall that $\log_{|\Omega|}|G|\leq b_\Omega(G)$.  
 By Lemmas~\ref{lemma:Tbound} and~\ref{lemma:keybound}(\ref{eqn:key}), $b_\Omega(G)=\lceil\log_{|\Omega|}|G|\rceil+\varepsilon=\lceil\log_{|T|}d\rceil+\delta$
 for some $\varepsilon,\delta\in \{0,1,2,3\}$ where $\varepsilon\leq \delta$. In particular, the first claim in Theorem~\ref{thm:pyber} holds.

Now suppose that $G$ is primitive on $\Omega$. Then $\core_P(Q)=1$ by Lemma~\ref{lemma:twmax}, so $\delta\leq \varepsilon+1$ by Lemma~\ref{lemma:keybound}(\ref{eqn:primkey}). It remains to prove that $\varepsilon\leq 2$. 
If $P$ is quasiprimitive on $\bm{k}$, then $b_\Omega(G)=2$ by Theorem~\ref{thm:QP}, so $\varepsilon\leq 2$. Thus we may assume that $P$ is not quasiprimitive, so  Theorem~\ref{thm:blowup}(ii) holds for $G$. Let $H:=\overline{G}$ and $\Delta:=\overline{B}$ as in Theorem~\ref{thm:blowup}(ii). By Theorem~\ref{thm:blowup}, $H$ is primitive of twisted wreath type with a quasiprimitive top group, so $b_\Delta(H)=2$ by Theorem~\ref{thm:QP}. Moreover, by Theorem~\ref{thm:blowup}(ii),  $G$ and $H$ satisfy the conditions of Theorem~\ref{thm:expgeneral}, so $\varepsilon\leq 2$, as desired.
\end{proof}

\begin{proof}[Proof of Theorem~\emph{\ref{thm:exp}}]
Recall that $G=H\wr K$ where $H$ is a primitive permutation group on $\Delta$ of twisted wreath type with socle $T^m$, and $K$ is a transitive subgroup of $S_r$ where $r\geq 2$. 
By Lemma~\ref{lemma:QPprod},  $G$ is a primitive permutation group of twisted wreath type on $\Omega=\Delta^r$. 
 By Theorem~\ref{thm:twONan} and Lemma~\ref{lemma:twmax},  $T$ is a section of the top group $R$ of $H$, and $[H:T^m]=|R|$, so $[H:T^m]\geq |T|>48$. By assumption, $R$ is quasiprimitive, so  $b_\Delta(H)=2$ by  Theorem~\ref{thm:QP}. Thus the result follows from   Theorem~\ref{thm:expgeneral}.
\end{proof}

\section{Examples}
\label{s:exp}

The following is~\cite[Example~4.8]{Bad1993}, which demonstrates (in particular) that for any finite non-abelian simple group $T$, there are primitive permutation groups of twisted wreath type with socle $T^k$  for infinitely many $k$, and for all of these  groups, the top group is quasiprimitive. 

\begin{example}
\label{exp:diag}
Let $T$ be a finite non-abelian simple group, and let $\ell\geq 2$. Let $P$ be a quasiprimitive permutation group of diagonal type with socle $A:= T^\ell$, where $A$ is the unique minimal normal subgroup of $P$. (These are precisely the  groups of type III(a)(i) in Table~\ref{tab:ONS}.)
We may view $P$ as a subgroup of $S_k$ where $k:=|T|^{\ell-1}$, in which case $Q:=P_1$ is isomorphic to a subgroup of $\Aut(T)\times S_k$, and $T\simeq A\cap Q\unlhd Q$. For each $x\in Q$, we may define $\varphi(x)\in \Aut(A\cap Q)$ to be conjugation by $x$,  and since $\core_P(Q)=1$, it follows that 
$G:=T\twr_\varphi P$ is a quasiprimitive permutation group. 
In fact, $G$ is primitive by Theorem~\ref{thm:twmax}  (see~\cite[Example 4.8 and Lemma~4.4]{Bad1993} for details). Note that $P$ is primitive on $\bm{k}$ if and only if the permutation group $P^{\bm{\ell}}$ induced by the action of $P$ on the $\ell$  simple factors of $A$ is primitive. Also, for any transitive $R\leq S_\ell$, we can choose $P$ so that $P^{\bm{\ell}}=R$. Thus the top group of $G$ can be either primitive, or  quasiprimitive and  not primitive. 
\end{example}

Using Example~\ref{exp:diag}, we construct a large family of primitive permutation groups of twisted wreath type with base size $2$ whose top groups are not quasiprimitive. 

\begin{example}
\label{exp:topnotQP}
Let $T$ be any finite non-abelian simple group, and let $\ell\geq 2$. Let $m:=|T|^{\ell-1}$.  
By Example~\ref{exp:diag}, there is a primitive permutation group $H$  of twisted wreath type with socle $T^m$ and top group $S$ where $S$ is quasiprimitive on $\bm{m}$  with socle $T^\ell$. Let $K$ be any transitive permutation group of degree $r\geq 2$ such that    either $r\leq 2^m$, or $K$ is soluble or  semiprimitive (but not $S_r$ or $A_r$). Then $G:=H\wr K$ is a primitive permutation group of twisted wreath type with socle $T^{mr}$ by Lemma~\ref{lemma:QPprod}. The top group of $G$ is $P:=S\wr K$ in its imprimitive action on $\bm{m}\times\bm{r}$, and $P$ is not quasiprimitive since $S^r$ fixes the $r$ blocks of imprimitivity.
By Theorem~\ref{thm:summary}, $d_{\bm{m}}(S)=2$ and $d_{\bm{r}}(K)\leq 2^m$, so by  Lemma~\ref{lemma:distblock}, $d_{\bm{m}\times\bm{r}}(P)\leq 4$. 
As we saw in the proof of Theorem~\ref{thm:QPgeneral}, $\omega_T(\Aut(T))\geq 4$. Thus $b(G)=2$ by  Lemma~\ref{lemma:hboundplus}.
\end{example}

Next we show that for any finite non-abelian simple group $T$, there is a primitive permutation group $G$ of twisted wreath type with socle $T^k$ for some $k$ such that the top group $P$ of $G$ is primitive and  almost simple. In Example~\ref{exp:AS}, we construct $G$ for an almost simple group $P$ with certain properties, and then in Lemma~\ref{lemma:AS}, we show that some $P$ with these properties  can be constructed for every $T$.

\begin{example}
\label{exp:AS}
Let $P$ be an almost simple primitive permutation group on $\bm{k}$  such that $Q:=P_1$  is almost simple with socle $T$. For each $x\in Q$,  define $\varphi(x)\in \Aut(T)$ to be conjugation by $x$. Since $\core_P(Q)=1$, it follows that 
$G:=T\twr_\varphi P$ is a quasiprimitive permutation group. To show that $G$ is primitive, it suffices to show that conditions (i)--(iii) of  Theorem~\ref{thm:twmax} hold with $M=P$. Let $U:=\ker(\varphi)=1$ and $V:=\varphi^{-1}(\Inn(T))=T$. Clearly (i) holds. Since $P$ is a primitive permutation group, $Q=N_P(T)$, so (ii) holds. 
For (iii), let $R$ be a non-trivial subgroup of $P$ that is normalised by $Q$ for which $R\cap T=1$. Then $R\cap Q=1$ and $R\unlhd P$, so $R$ is a non-trivial regular normal subgroup of $P$, contradicting the O'Nan-Scott theorem.  Thus (iii) holds, so $G$ is primitive by  Theorem~\ref{thm:twmax}.
To see that such a group $P$ exists for any non-abelian simple group $T$,  see Lemma~\ref{lemma:AS}.
\end{example}

\begin{lemma}
\label{lemma:AS}
Let $T$ be a non-abelian simple group. Then there is an almost simple primitive permutation group  whose point stabiliser is  almost simple with socle $T$.
\end{lemma}

\begin{proof}
We claim that there is an almost simple group $G$ with socle $T$ and an almost simple group $H$ with socle $S$ such that $G$ is maximal in $H$ and $S\neq T$. If so, then $G$ is core-free in $H$, so $H$ is a primitive permutation group  on the set of cosets of $G$, and we are done. If $T=A_m$, then we may take  $(G,H)=(S_m,S_{m+1})$, so we may assume that $T\not\simeq A_m$ for $m\geq 5$. 
Choose a maximal subgroup $M$ of $T$,  and let $n:=[T:M]$.  Then we may view $T$ as  a primitive subgroup of $S_n$, and $T$ is not regular. Let $N:=N_{S_n}(T)$. Then $N$ is almost simple with socle $T$, and  $NA_n$ has socle $A_n$.
If $N$ is maximal in $NA_n$, then  we may take $(G,H)=(N,NA_n)$. Otherwise, there is a group $K$ such that $N<K<NA_n$ and $N$ is maximal in $K$. If $K$ is almost simple with socle $S$, then $S\neq T$, so we may take $(G,H)=(N,K)$. 
Otherwise, $K$ is a primitive subgroup of $S_n$ that is not almost simple. By~\cite{LiePraSax1987}, $T$ is   $\PSL_2(7)$, $M_{12}$ or $\Sp_4(q)$ where $q$ is even. By~\cite{LiePraSax1987}, we may take $(G,H)$ to be $(\PSL_2(7){:}2,S_8)$, $(M_{12},A_{12})$  or $(\Sp_4(q),\SL_4(q))$, respectively.
\end{proof}

We finish this section with two more examples and a remark.

\begin{example}
\label{exp:Sktop}
In Theorems~\ref{thm:QPgeneral} and~\ref{thm:primprobgeneral}, we proved that if  Hypothesis~\ref{hyp:tw} holds where  $P$ is $S_k$ or $A_k$ and $\Inn(T)\leq \varphi(Q)$, then $b_B(G)=2$ and the proportion of pairs of points from $B$ that are bases for $G$ tends to 1 as $|G|\to \infty$.
The assumption that $\Inn(T)\leq \varphi(Q)$ cannot be removed. For example, by Remark~\ref{remark:lotsQP},  quasiprimitive permutation groups $G$ of twisted wreath type with socle $A_5^k$ and top group $S_k$  exist for all $k\geq 2$, but   $b_B(G)\geq\log_{60}k$  by Theorem~\ref{thm:pyber}.
\end{example}

\begin{example}
\label{exp:SkwrSr}
Assume Hypothesis~\ref{hyp:tw} where $G$ is primitive on $B$, so $\core_P(Q)=1$ by Lemma~\ref{lemma:twmax}.  
Let $K$ be a transitive subgroup of $S_r$ where $r\geq 2$. By Lemma~\ref{lemma:QPprod}, $G\wr K$ is a primitive permutation group of twisted wreath type on $B^r$. 
Moreover, the top group of $G\wr K$ is $P\wr K$ in its imprimitive action on $\bm{k}\times \bm{r}$. In particular, by Example~\ref{exp:AS}, we may take $T=A_{k-1}$, $P=S_k$ and $K=S_r$ where $k\geq 6$, in which case $G_{k,r}:=G\wr K$ has top group $S_k\wr S_r$. 
By Theorem~\ref{thm:exp}, 
$$
b(G_{k,r})=\left\lceil \frac{\log r}{k\log|A_{k-1}|}\right\rceil+\delta_{k,r}
$$
for some $\delta_{k,r}\in \{1,2\}$. 
In particular,  $b(G_{k,r})$ can be arbitrarily large, but also very small. For example, $b(G_{k,r})\in \{2,3\}$ for all $r\leq ((k-1)!/2)^k$.  
In fact, $b(G_{k,k})=2$ for $k\geq 9$ by  
Lemma~\ref{lemma:hboundplus} since    
$d_{\bm{k}\times\bm{k}}(S_k\wr S_k)=k+1$ by~\cite[Theorem~2.3]{Cha2006} (see~\S\ref{s:dist}), and by the proof of Lemma~\ref{lemma:QPSkAk}, $k+1\leq \omega_{A_{k-1}}(S_{k-1})$ for $k\geq 9$. 
Note that by~\cite[Proposition~5.5.1]{FawThesis}, if $X$ is a primitive permutation group of twisted wreath type with top group $S_k\wr S_r$ acting on  $\bm{k}\times\bm{r}$, then $X$ is permutation isomorphic to $G_{k,r}$. 
\end{example}

\begin{remark}
\label{remark:notsuff}
In the proof of Theorem~\ref{thm:QPgeneral} and Examples~\ref{exp:topnotQP} and~\ref{exp:SkwrSr}, we used 
 the  following consequence of  Lemma~\ref{lemma:hbound} (cf. Lemma~\ref{lemma:hboundplus}): 
if Hypothesis~\ref{hyp:tw} holds where $\core_P(Q)=1$ and $d_{\bm{k}}(P)\leq \omega_T(\varphi(Q))$, then $b_B(G)=2$. However,   
 the converse does not hold. 
By Example~\ref{exp:AS}, there is a primitive group $G$ of twisted wreath type with socle $A_{k-1}^k$ and top group $S_k$ where $k\geq 6$. Here Hypothesis~\ref{hyp:tw} holds with $P=S_k$, $T=A_{k-1}$ and  $\varphi(Q)=S_{k-1}$.  Now   $b_B(G)=2$ by Lemma~\ref{lemma:QPSkAk}, but if $k=6$ or $7$, then $\omega_T(\varphi(Q))=4$ or $6$, respectively, so $\omega_T(\varphi(Q))<k=d_{\bm{k}}(P)$.
\end{remark}

\bibliographystyle{acm}
\bibliography{jbf_references}
\end{document}